\documentclass{article}
\usepackage{cancel}
\usepackage{graphicx}
\usepackage{subcaption}
\usepackage{adjustbox}

\usepackage{algorithmicx}
\usepackage[ruled,vlined]{algorithm2e}
\usepackage{algpseudocode,float}
\algblock{ParFor}{EndParFor}

\usepackage{mathrsfs}  
\usepackage{sidecap}
\usepackage{amsfonts}
\usepackage{amssymb}
\usepackage{amsmath}
\usepackage{amsthm}
\usepackage{comment}
\usepackage{xcolor}
\usepackage{relsize}
\usepackage{multirow}
\usepackage{multicol}
\usepackage{diagbox}
\usepackage{overpic}
\usepackage{authblk}

\newcommand{\R}{{\mathbb R}}

\DeclareMathOperator*{\argmin}{arg\,min}
\newcommand{\STAB}[1]{\begin{tabular}{@{}c@{}}#1\end{tabular}}

\newtheorem{lem}{Lemma}
\newtheorem{cor}{Corollary}
\newtheorem{proposition}{Proposition}
\usepackage{cleveref}
\usepackage{autonum}
\date{}
\begin{document}
\title{Residual whiteness principle for automatic parameter selection in $\ell_2$-$\ell_2$ image super-resolution problems }
%
%
\author[1]{Monica Pragliola\thanks{monica.pragliola2@unibo.it}}
\author[2]{Luca Calatroni\thanks{calatroni@i3s.unice.fr}}
\author[1]{Alessandro Lanza\thanks{alessandro.lanza2@unibo.it}}
\author[1]{Fiorella Sgallari\thanks{fiorella.sgallari@unibo.it}}
\affil[1]{Department of Mathematics, University of Bologna, Italy}
\affil[2]{CNRS, UCA, INRIA, Morpheme, I3S, Sophia-Antipolis, France}

%
%
%
\maketitle              


\begin{abstract}
We propose an automatic parameter selection strategy for variational image super-resolution of blurred and down-sampled images corrupted by additive white Gaussian noise (AWGN) with unknown standard deviation. By exploiting particular properties of the operators describing the problem in the frequency domain, our strategy selects the optimal parameter as the one optimising a suitable residual whiteness measure. 
Numerical tests show the effectiveness of the proposed strategy for generalised $\ell_2$-$\ell_2$ Tikhonov problems.
\end{abstract}
%
\section{Introduction}
The problem of single-image Super-Resolution (SR) consists in finding a high-resolution (HR) image starting from low-resolution (LR) blurred and noisy data. The huge number of applications which benefits from the recovery of HR information, ranging from remote sensing to biomedical imaging, motivates the large amount of research still ongoing in this field. 

Mathematically, the problem can be described as follows. Let $\mathbf{X}\in \R^{N_r \times N_c}$ denote the original HR image, with $\mathbf{x}=\mathrm{vec}(\mathbf{X})\in \R^{N}$, $N=N_r N_c$, being its vectorisation. The process describing the mapping from HR to LR data can be described by the following linear observation model
\begin{equation}
\label{eq:lin_model}
   \mathbf{b} = \mathbf{S K x + e}\,,\,\quad \text{with }\mathbf{e}\text{ realisation of  }\mathbf{E}\sim\mathcal{N}(0,\sigma^2 \mathbf{I}_n)\,,
\end{equation}
where $\mathbf{b},\mathbf{e}\in \R^n$, $n=n_r n_c$, are the vectorised observed and noise image, respectively, both consisting of ${n_r\times n_c}$ pixels,  $\mathbf{S}\in\R^{n\times N}$ is the down-sampling operator inducing a pixel decimation with factor $d_r$ and $d_c$ along the rows and the columns of $\mathbf{X}$, respectively - i.e., $N_r=n_rd_r$, $N_c=n_cd_c$ - $\mathbf{K}\in\R^{N\times N}$ represents a space-invariant blurring operator, $\mathbf{I}_n \in \R^{n \times n}$ denotes the $n$-dimensional identity matrix and $\mathbf{E}$ is an $n$-variate Gaussian-distributed random vector with zero mean and scalar covariance matrix, with $\sigma$ indicating the (unknown) noise standard deviation. We set $d=d_r d_c$, so that $N = nd$.

To overcome the ill-posedness of problem \eqref{eq:lin_model}, 
one can seek an estimate $\mathbf{x}^*$ of $\mathbf{x}$ by minimising a suitable cost function $\mathcal{J}:\R^N \to \R^+$. In this work, we consider in particular a generalised $\ell_2$-$\ell_2$ Tikhonov-regularised model of the form
\begin{equation}
\label{eq:l2l2}
\mathbf{x^*}(\mu)\;{=}\;\arg\min_{\mathbf{x}\in\R^{N}}\left\{\mathcal{J}(\mathbf{x};\mu)\;{:=}\;\frac{\mu}{2}\|\mathbf{SKx}-\mathbf{b}\|_2^2+\frac{1}{2}\|\mathbf{Lx}-\mathbf{v}\|_2^2\right\}\,,
\end{equation}
where the operator $\mathbf{L}\in\R^{M\times N}$ and the vector $\mathbf{v}\in\R^{M}$ are known. The data term $\|\mathbf{SKx}-\mathbf{b}\|_2^2$ encodes the AWGN assumption on $\mathbf{e}$, while the regularisation term $ \|\mathbf{Lx}-\mathbf{v}\|_2^2$ encodes prior information on the unknown target. Finally, the  \emph{regularisation parameter} $\mu \in \R^+_*$ in (\ref{eq:l2l2}) 
balances the action of the fidelity against regularisation; its choice is of crucial importance for  high quality reconstructions. \\
\indent When $\mathbf{S}=\mathbf{I}_N$, under general assumptions - see (A3)-(A4) \mbox{in Sec.~\ref{sec:ass} - the} problem in \eqref{eq:l2l2} can be solved very efficiently. However, the presence of a non-trivial $\mathbf{S}$ makes the computation of the least-squares solution very costly. In \cite{FSR}, upon a specific choice of $\mathbf{S}$, the authors proposed an efficient strategy for the solution of \eqref{eq:l2l2}, for which Generalised Cross Validation \cite{gcv} is used to select the optimal $\mu$. This is known to be impractical for large-scale problems \cite{RIP}.

\indent A popular strategy which aims at overcoming the downsides of empirical parameter selection rules while exploiting the information available on the noise corruption is the celebrated discrepancy principle (DP) (see \cite{hansen,Chen2013} for general problems and \cite{ItDiscr2015} for applications to super-resolution problems), which can be formulated as follows:
\begin{equation}\label{eq:dp}
    \text{select }\mu = \mu^*\text{ such that }\|\mathbf{r}^*(\mu^*)\|_2 = \|\mathbf{SKx}^*(\mu)-\mathbf{b}\|_2 = \tau\sqrt{n}\sigma\,,
\end{equation}
with $\mathbf{x}^*(\mu)$ being the solution of \eqref{eq:l2l2} and $\tau$ denoting the discrepancy coefficient. When $\sigma$ is known, $\tau$ is set equal to $1$, otherwise a value slightly greater than $1$ is typically chosen to avoid noise under-estimation. Clearly, in real world applications an accurate estimate of $\sigma$ is not available, which often limits the applicability of DP strategies.\\
\indent Recently, in the context of image restoration problems, a number of works has focused on the design of variational models explicitly exploiting in their formulations the assumed whiteness of the corrupting noise - see, e.g., \cite{LMSS,riot}. Based on these promising results, in \cite{etna}, the authors propose a strategy named \emph{residual whiteness principle} (RWP), that relies on the whiteness property of the noise to properly set the regularisation parameter $\mu$. The RWP automatically selects a value for $\mu$ that maximises the whiteness of the residual image $\mathbf{r}^*(\mu)=\mathbf{SKx}^*(\mu)-\mathbf{b}$, or equivalently minimises the squared Euclidean norm of the normalised auto-correlation of $\mathbf{r}^*(\mu)$. 
The RWP has there been applied to the automatic selection of $\mu$ in Tikhonov-regularised least squares problems which are frequently encountered in  iterative  \emph{alternating direction method of multipliers} (ADMM) optimisation frameworks when used to larger classes of non-smooth regularisation models.\\
\indent In this paper, we extend the results obtained in \cite{etna} to SR problems of the form \eqref{eq:l2l2}. 
As in \cite{etna}, the proposed strategy can be easily extended to models more general than the one in \eqref{eq:l2l2}.




\subsection{Notations, preliminaries and assumptions}
\label{sec:ass}
In the following, for $c\in\mathbb{C}$ we use $\overline{c},|c|$  to indicate the conjugate and the modulus of $c$, respectively. We denote by $\mathbf{F},\mathbf{F}^H$ the 2D Fourier transform and its inverse, respectively. For any $\textbf{v}\in\mathbb{R}^N$ and any $\mathbf{A}\in\mathbb{R}^{N\times N}$, we use the notations $\tilde{\mathbf{v}}=\mathbf{F}\mathbf{v}$ and $\tilde{\mathbf{A}}=\mathbf{F}\mathbf{A}\mathbf{F}^H$ to denote the action of the 2D Fourier transform operator $\mathbf{F}$ on vectors and matrices, respectively. Given a permutation matrix $\mathbf{P}\in\mathbb{R}^{N\times N}$, we denote by $\hat{\mathbf{v}}=\mathbf{P}\tilde{\mathbf{v}}$ and by $\hat{\mathbf{A}}=\mathbf{P}\tilde{\mathbf{A}}\mathbf{P}^T$ the action of $\mathbf{P}$ on the Fourier-transformed vector $\tilde{\mathbf{v}}$ and matrix $\tilde{\mathbf{A}}$, respectively. Finally, by $\check{\mathbf{A}}$ we denote the product $\check{\mathbf{A}}=\mathbf{P}\tilde{\mathbf{A}}^H\mathbf{P}^T$, i.e. the action of $\mathbf{P}$ on $\tilde{\mathbf{A}}^H$.

We recall some results that will be useful in the following discussion and a well-known property of the Kronecker product `$\otimes$'.

\begin{lem}[\cite{lemma}]\label{lem:FSSH}
Let $\mathbf{J}_d\in\R^{d\times d}$ denote a matrix of ones. We have:
\begin{equation}\label{eq:FSSH}
\widetilde{\mathbf{S}^H\mathbf{S}} = \frac{1}{d}(\mathbf{J}_{d_r}\otimes\mathbf{I}_{n_r}) \otimes (\mathbf{J}_{d_c}\otimes\mathbf{I}_{n_c})\,.
\end{equation}
\end{lem}

\begin{lem}\label{lem:kron}
Let $\mathbf{A},\mathbf{B},\mathbf{C},\mathbf{D}$ be given matrices such that $\mathbf{AC},\mathbf{BD}$ exist. We have:
\begin{equation}\label{eq:kron}
    (\mathbf{A}\otimes\mathbf{B})
    (\mathbf{C}\otimes\mathbf{D}) = 
    (\mathbf{AC}\otimes\mathbf{BD})\,.
\end{equation}
\end{lem}

\begin{lem}[Woodbury formula] \label{lem:woodbury}
Let $\mathbf{A}_1, \mathbf{A}_2,\mathbf{A}_3,\mathbf{A}_4$ matrices and let $\mathbf{A}_1$ and $\mathbf{A}_3$ be invertible. Then, the following inversion formula holds:
\begin{eqnarray}\label{eq:woodbury}
(\mathbf{A}_1+\mathbf{A}_2\mathbf{A}_3\mathbf{A}_4)^{-1}=\mathbf{A}_1^{-1}+\mathbf{A}_1^{-1}\mathbf{A}_2(\mathbf{A}_3^{-1}+\mathbf{A}_4\mathbf{A}_1^{-1}\mathbf{A}_2)^{-1}\mathbf{A}_4\mathbf{A}_1^{-1}.
\end{eqnarray}

\end{lem}


The results recalled and proposed in this paper rely on the following assumptions on the image formation model and on the linear operators $\mathbf{S,K,L}$.
\begin{itemize}
    \item[(A1)] The original image $\mathbf{X}$ is assumed to be square, i.e.  $N_r=N_c$, and $d_c=d_r$.
    \item[(A2)] The conjugate transpose $\mathbf{S}^H\in\R^{N\times n}$ of the down-sampling operator interpolates the decimated image with zeros, and $\mathbf{SS}^H = \mathbf{I}_n$.
    \item[(A3)] The matrices $\mathbf{S,K}$ and $\mathbf{L}$ in \eqref{eq:l2l2} are such that $\mathrm{null}(\mathbf{SK}) \cap \mathrm{null}(\mathbf{L}) = \mathbf{0}_N$, with $\mathbf{0}_N$ denoting the $N$-dimensional null vector.
    \item[(A4)] As a consequence of the space-invariance of the blur, the matrix $\mathbf{K}$ represents a 2D discrete convolution operator. Also the regularisation matrix $\mathbf{L}$ is required to represent a 2D convolutional operator, so that $\mathbf{K}$ and $\mathbf{L}$ can be diagonalised by the 2D discrete Fourier transform. In formula:
    \begin{equation}\label{eq:KL_diag}
    \mathbf{K} = \mathbf{F}^H\mathbf{\Lambda}\mathbf{F}\quad\text{and}\quad   \mathbf{L} = \mathbf{F}^H\mathbf{\Gamma}\mathbf{F}\,,\quad\text{with}\quad \mathbf{F}^H\mathbf{F}=\mathbf{F}\mathbf{F}^H = \mathbf{I}_N\,,
\end{equation}
where $\mathbf{\Lambda},\mathbf{\Gamma}\in\mathbb{C}^{N\times N}$ are diagonal matrices defined by
\begin{equation}
    \mathbf{\Lambda} = \mathrm{diag}(\tilde{\lambda}_1,\ldots,\tilde{\lambda}_N)\,,\quad  \mathbf{\Gamma} = \mathrm{diag}(\tilde{\gamma}_1,\ldots,\tilde{\gamma}_N)\,.
\end{equation}
\end{itemize}



Notice that assumption (A3) guarantees the existence of global minimisers for the cost function $\mathcal{J}(\cdot;\mu):\R^N\to \R^+$ in \eqref{eq:l2l2}.

\section{Residual whiteness principle}\label{sec:RWP}

Let us consider the noise realisation $\mathbf{e}$ in (\ref{eq:lin_model}) in its original $n_r \times n_c$ matrix form:
\begin{equation}
\mathbf{e}  \,\;{=}\;\, \left\{ e_{i,j} \right\}_{(i,j) \in \Omega}, \quad 
\Omega \;{:=}\; \{ 0 , \,\ldots\, , n_r-1 \} \times \{ 0 , \,\ldots\, , n_c-1 \}.
\end{equation}
%
The \emph{sample auto-correlation} $a: \R^{n_r \times n_c} \to \R^{(2n_r-1) \times (2n_c-1)}$ of realisation $\mathbf{e}$ is
\begin{equation}
\label{eq:theta}
a(\mathbf{e}  ) {=} 
\left\{ a_{l,m}(\mathbf{e}  ) \right\}_{(l,m) \in \mathrm{\Theta}}, 
\, 
\mathrm{\Theta}{:=} \{ -(n_r -1) ,\ldots , n_r - 1 \} \times \{ -(n_c -1) , \ldots , n_c - 1 \},
\end{equation}
with each scalar component $a_{l,m}(\mathbf{e}  ): \R^{n_r \times n_c} \to \R$ given by
\begin{equation}
a_{l,m}(\mathbf{e}  )
{=}
\frac{1}{n} \:
\big( \, \mathbf{e}   \,\;{\star}\;\: \mathbf{e}   \, \big)_{l,m}
{=}
\frac{1}{n} \: \big( \, \mathbf{e}   \,\;{\ast}\;\: \mathbf{e}  ^{\prime} \, \big)_{l,m} 
{=}
\frac{1}{n} \!\!
\sum_{\;(i,j)\in\,\mathrm{\Omega}} \!
e_{i,j} \, e_{i+l,j+m} \, ,
\; (l,m) \in \mathrm{\Theta} \, , 
\label{eq:n_ac}
\end{equation}
where index pairs $(l,m)$ are commonly called \emph{lags}, $\,\star\:$ and $\,\ast\,$ denote the 2-D discrete correlation and convolution operators, respectively, and where
$\mathbf{e}  ^{\prime}(i,j) = \mathbf{e}  (-i,-j)$.
The noise realisation $\mathbf{e}$ is padded with at least
$n_r-1$ samples in the vertical direction and $n_c-1$ samples in the horizontal direction
by assuming periodic boundary conditions, such that $\,\star\:$ and $\,\ast\,$ in (\ref{eq:n_ac}) denote 2-D circular correlation
and convolution, respectively. This allows to consider only lags 
%
\begin{equation}
(l,m) \in \overline{\mathrm{\Theta}} 
\;{:=}\; \{ 0, \,\ldots\, , n_r - 1\} \times \{ 0 , \,\ldots\, , n_c - 1 \}.
\label{eq:Theta_bar}
\end{equation}

If the corruption $\mathbf{e}$ in (\ref{eq:lin_model}) is the realisation of a white Gaussian noise process - as in our case - it is well known that as $n\to + \infty$, the sample auto-correlation $a_{l,m}(\mathbf{e})$ vanishes for all $(l,m)\neq(0,0)$, while $a_{0,0}(\mathbf{e})=\sigma^2$ - see, e.g., \cite{LMSS}. 

The DP exploits only the information at lag $(0,0)$. In fact, the standard deviation recovered by the residual image is required to be equal to $\sigma$. Imposing whiteness of the restoration residual by constraining the residual auto-correlation at non-zero lags to be small is a much stronger requirement. 

In \cite{etna}, the authors introduce the following non-negative scalar measure of whiteness $\mathcal{W}: \R^{n_r \times n_c} \to \R^+$ of noise realisation $\mathbf{e}$:
\begin{equation}
\mathcal{W}(\mathbf{e}  ) \,\;{:=}\;\, 
\left\|\,\mathbf{e}   \,\;{\star}\;\, \mathbf{e}\,\right\|_2^2 / 
\left\|\mathbf{e}  \right\|_2^4
\,\;{=}\;\,
\widetilde{\mathcal{W}}(\tilde{\mathbf{e}  }) \, ,
\label{eq:NAC}
\end{equation}
where $\| \cdot \|_2$ denotes the Frobenius norm, while the second equality  comes from Proposition \ref{prop:WFour} below, with
$\widetilde{\mathcal{W}}: \mathbb{C}^{n_r \times n_c} \to \R^+$ the function defined in (\ref{eq:WFour}). Notice that the presence of the denominator in the function in \eqref{eq:NAC} makes the whiteness principle completely independent of the noise level.

\begin{proposition}
\label{prop:WFour}
Let $\mathbf{e}   \in \R^{n_r \times n_c}$ and $\tilde{e} \in \mathbb{C}^{n_r \times n_c}$. Then, under the assumption of periodic boundary conditions for $\mathbf{e}$, the function $\mathcal{W}$ defined in (\ref{eq:NAC}) satisfies:
\begin{equation}
\mathcal{W}(\mathbf{e}) 
\,\;{=}\;\, 
\widetilde{\mathcal{W}}(\tilde{\mathbf{e}})
\,\;{:=}\;\:\, 
\displaystyle{
\sum_{(l,m) \in \overline{\mathrm{\Theta}}} 
\left| \tilde{e}_{l,m} \right|^4
} \Big/ 
\displaystyle{
\Bigg(
\sum_{(l,m) \in \overline{\mathrm{\Theta}}}
\left| \tilde{e}_{l,m} \right|^2
\Bigg)^2
}\,.
\label{eq:WFour}
\end{equation}
\end{proposition}

\section{RWP for super-resolution}
By now looking at (\ref{eq:l2l2}), we observe that the nearer the super-resolved image $\mathbf{x}^*(\mu)$ is to the original image $\mathbf{x}$, the closer the associated residual image $\mathbf{r}^*(\mu) = \mathbf{SKx}^*(\mu) - \mathbf{b}$ is to the white noise realisation $\mathbf{e}$ in (\ref{eq:lin_model}) and, hence, the whiter is the residual image according to the scalar measure in (\ref{eq:NAC}).


This motivates the choice of the RWP for automatically selecting the regularisation parameter $\mu$ in variational models of the form (\ref{eq:l2l2}), which reads:
\begin{equation}
\text{Select }\mu=\mu^*\text{  s.t.  }
\mu^* \:{\in}\: 
\arg\min_{\mu \in \R^+_*} 
\!W(\mu)  := \mathcal{W}\left(\mathbf{r}^*(\mu)\right)\,,
\label{eq:prob_lambda}
\end{equation}
where 
the scalar non-negative cost function $W: \R^+_* \to \R^+$ in (\ref{eq:prob_lambda}), from now on referred to as the \emph{residual whiteness function}, takes the following form:
\begin{equation}
W(\mu) \,\;{=}\;\, 
\left\| \, \mathbf{r}^*(\mu) \,\;{\star}\;\, \mathbf{r}^*(\mu) \, \right\|_2^2 / \left\|\mathbf{r}^*(\mu)\right\|_2^4\,.
\label{eq:Wfun_freq_a}
\end{equation}
Let us now give a closer look to the function in \eqref{eq:Wfun_freq_a}. First, we observe
\begin{equation}
    \mathbf{r}^{*}(\mu)= \mathbf{S K x^*}(\mu)-\mathbf{b}=\mathbf{S K x}^*(\mu)-\mathbf{SS}^H\mathbf{b} =\mathbf{S}\mathbf{r}^*_H(\mu)\,,
\end{equation}
where $\mathbf{r}_H^*(\mu)=\mathbf{Kx}^*(\mu) - \mathbf{b}_H$ is the high-resolution residual, while $\mathbf{b}_H=\mathbf{S}^H\mathbf{b}$. The denominator in \eqref{eq:Wfun_freq_a} can be thus expressed as follows
\begin{equation}
\label{eq:deno last}
\|\mathbf{r}^*(\mu)\|_2^4   = \|\mathbf{Sr}^*_H(\mu)\|_2^4 = \|\mathbf{S}^H\mathbf{Sr}_H^*(\mu)\|_2^4 = \|\mathbf{F}^H(\mathbf{F}\mathbf{S}^H\mathbf{S}\mathbf{F}^H)\mathbf{F}\mathbf{r}^*_H(\mu)\|_2^4\,,
\end{equation}
where the second equality comes from recalling that $\mathbf{S}^H$ interpolates $\mathbf{Sr}^*_H(\mu)$ with zeros giving null contribution when computing the norm.
From Lemma \ref{lem:FSSH} and by applying the Parseval's theorem, we get the following chain of equalities:
\begin{align}
\|\mathbf{r}^*(\mu)\|_2^4 =& \left\|(1/d)\mathbf{F}^H( \mathbf{J}_{d_r} \otimes \mathbf{I}_{n_r}) \otimes(\mathbf{J}_{d_c} \otimes \mathbf{I}_{n_c}) \tilde{\mathbf{r}}^*_H(\mu) \right\|_2^4\\
\label{eq:norm_rH}
=&  \left\|(1/d)( \mathbf{J}_{d_r} \otimes \mathbf{I}_{n_r}) \otimes(\mathbf{J}_{d_c} \otimes \mathbf{I}_{n_c})  \tilde{\mathbf{r}}^*_H(\mu) \right\|_2^4\,.
\end{align}
The non-zero entries of the matrix introduced in Lemma \ref{lem:FSSH}, which are all equal to $1$, are arranged along replicated patterns; this particular structure can be exploited by considering a permutation matrix $\mathbf{P}\in\R^{N\times N}$ such that:
\begin{equation}  \label{eq:permutation}
 \mathbf{P}\left[( \mathbf{J}_{d_r} \otimes \mathbf{I}_{n_r}) \otimes(\mathbf{J}_{d_c} \otimes \mathbf{I}_{n_c})\right] \mathbf{P}^T =(\mathbf{I}_n\otimes \mathbf{J}_d)\,. 
\end{equation}
The designed permutation acts on the matrix of interest by gathering together the replicated rows and columns. In Fig.~\ref{fig:PERM}, we show the structure of the matrix in \eqref{eq:FSSH} and of the permuted matrix in  \eqref{eq:permutation} for $n_r{=}n_c{=}3$ and $d_r{=}d_c{=}2$.

\begin{SCfigure}[]
\centering
  \begin{subfigure}[t]{0.28\textwidth}
    \centering
    \includegraphics[scale=0.22]{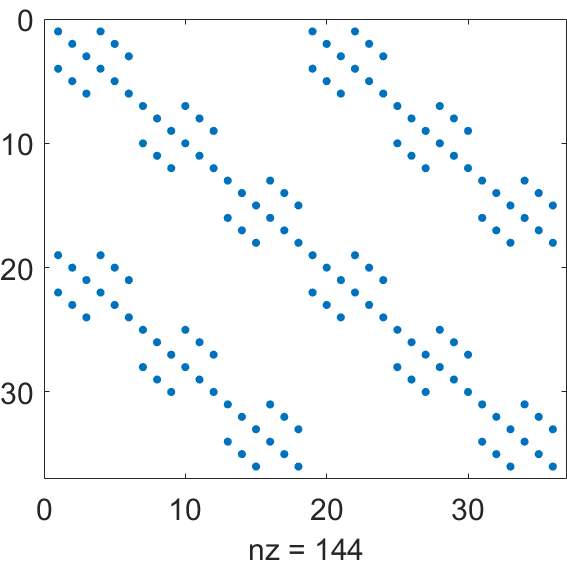}
    \label{fig:noperm}
    \end{subfigure}
    \begin{subfigure}[t]{0.28\textwidth}
    \centering
    \includegraphics[scale=0.22]{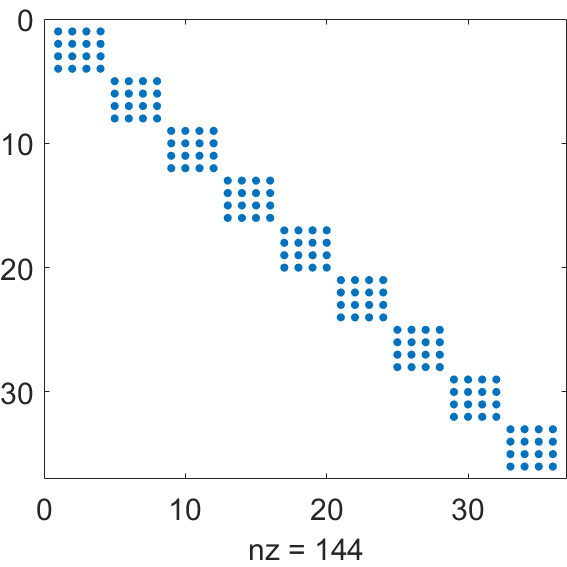}
        \label{fig:perm}
    \end{subfigure}
\caption{Structure of the matrix in \eqref{eq:FSSH} (left) and of the permutation induced by $\mathbf{P}$ (right) for $n_r{=}n_c{=}3$, $d_r{=}d_c{=}2$.}\label{fig:sidecap}
    \label{fig:PERM}
\end{SCfigure}

Hence, the expression in \eqref{eq:norm_rH} can be rewritten as
\begin{equation}
\|\mathbf{r}^*(\mu)\|_2^4 = \left\|\frac{1}{d}\mathbf{P}\left[( \mathbf{J}_{d_r} \otimes \mathbf{I}_{n_r}) \otimes(\mathbf{J}_{d_c} \otimes \mathbf{I}_{n_c})\right] \mathbf{P}^T\mathbf{P} \tilde{\mathbf{r}}^*_H(\mu) \right\|^4_2=  \left\|\frac{1}{d}( \mathbf{I}_{n} \otimes \mathbf{J}_{d})  \hat{\mathbf{r}}^*_H(\mu) \right\|^4_2\,,
\end{equation}
where
\begin{equation}
\label{eq:iota}
\left((\mathbf{I}_n\otimes \mathbf{J}_d)\hat{\mathbf{r}}^*_H(\mu) \right)_i = \sum_{j=0}^{d-1}    \left(\hat{\mathbf{r}}^*_H(\mu) \right)_{\iota + j}\,,\;\;\text{with  }\iota:=1+ \Big\lfloor \frac{i-1}{d}\Big\rfloor d\,,
\end{equation}
for every $i=1,\ldots N$. The denominator in \eqref{eq:Wfun_freq_a} can be thus expressed as
\begin{equation}
\|\mathbf{r}^*(\mu)\|_2^4 =\frac{1}{d^4} \left(\sum_{i=1}^{N}\left| \sum_{j=0}^{d-1}    \left(\hat{\mathbf{r}}^*_H(\mu) \right)_{\iota+ j}  \right|^2  \right)^2\,.
\end{equation}

Let us now consider the numerator of the function $W(\mu)$ in \eqref{eq:Wfun_freq_a}, which, based on the definitions of auto-correlation given in \eqref{eq:n_ac} and of $\mathbf{S}^H$, reads
\begin{equation}
    \|\mathbf{r}^*(\mu)\star \mathbf{r}^*(\mu)\|_2^2 = \|\mathbf{Sr}^*_{H}(\mu)\star \mathbf{Sr}^*_{H}(\mu)\|_2^2= \|\mathbf{S}^H\mathbf{Sr}^*_{H}(\mu)\star\mathbf{S}^H \mathbf{Sr}^*_{H}(\mu)\|_2^2\,.
    \end{equation}
By applying again the Parseval's theorem and the convolution theorem, we get
\begin{align}
 \|\mathbf{r}^*(\mu)\star \mathbf{r}^*(\mu)\|^2_2 =& 
 \|\mathbf{F}\left((\mathbf{S}^H\mathbf{Sr}^*_{H}(\mu))\star(\mathbf{S}^H \mathbf{Sr}^*_{H}(\mu))\right)\|^2_2\\
  =&\|\mathbf{F}(\mathbf{S}^H\mathbf{Sr}^*_{H}(\mu))\odot\overline{\mathbf{F}(\mathbf{S}^H \mathbf{Sr}^*_{H}(\mu)}\|^2_2\\
  =&\|\mathbf{F}(\mathbf{S}^H\mathbf{S})\mathbf{F}^H\mathbf{Fr}^*_{H}(\mu)\odot\overline{\mathbf{F}(\mathbf{S}^H\mathbf{S})\mathbf{F}^H \mathbf{Fr}^*_{H}(\mu)}\|^2_2\,,\label{eq:num}
\end{align}
where $\odot$ denotes the Hadamard matrix product operator. The expression in \eqref{eq:num} is manipulated by applying Lemma \ref{lem:FSSH} and the  permutation in \eqref{eq:permutation}, so as to give
    \begin{equation}\label{eq:num last}
           \|\mathbf{r}^*(\mu) \star \mathbf{r}^*(\mu)\|_2^2 =\frac{1}{d^4} \sum_{i=1}^{N}\left| \sum_{j=0}^{d-1}    \left(\hat{\mathbf{r}}^*_H(\mu)\right)_{\iota + j}  \right|^4\,.
    \end{equation}
Finally, plugging \eqref{eq:num last} and \eqref{eq:deno last} into \eqref{eq:Wfun_freq_a}, we get the following form for the whiteness measure $W(\mu)$ for a super-resolution problem
\begin{equation}
\label{eq:white_new}
W(\mu) = \left(\sum_{i=1}^N|w_i(\mu)|^4\right) /\left(\sum_{i=1}^N|w_i(\mu)|^2\right)^2\,,\; w_i(\mu)=\sum_{j=0}^{d-1}(\hat{\mathbf{r}}_H(\mu))_{\iota + j}\,.
\end{equation}

\subsection{RWP for $\ell_2$-$\ell_2$ problems in the form \eqref{eq:l2l2}}
\label{sec:rwpl2l2}
Here, we derive the analytical expression of the whiteness function $W(\mu)$ defined in \eqref{eq:white_new} when addressing Tikhonov-regularised least squares problems as the one in \eqref{eq:l2l2}. 
We start following \cite{FSR} to deduce an explicit and easily-computable expression of $\mathbf{x}^*(\mu)$. By optimality, we get:
\begin{equation}
\label{eq:x_gentik}
\mathbf{x}^*(\mu) = (\mu (\mathbf{SK})^H(\mathbf{SK})+\mathbf{L}^H\mathbf{L})^{-1}(\mu (\mathbf{SK})^H \mathbf{b} +\mathbf{L}^H\mathbf{v})\,,
\end{equation}    
which can be manipulated in terms of $\mathbf{F}$ and $\mathbf{F}^H$ to deduce
\begin{align}
     \mathbf{x}^*(\mu) = &(\mu \mathbf{F}^H\mathbf{F}\mathbf{K}^H\mathbf{F}^H\mathbf{F}\mathbf{S}^H\mathbf{S}\mathbf{F}^H\mathbf{F}\mathbf{K}\mathbf{F}^H\mathbf{F}+\mathbf{F}^H\mathbf{F}\mathbf{L}^T\mathbf{F}^H\mathbf{F}\mathbf{L}\mathbf{F}^H\mathbf{F})^{-1}(\mu \mathbf{K}^H\mathbf{S}^H \mathbf{b} +\mathbf{L}^H\mathbf{v})\\
     \label{eq:x_2}
    =& (\mu \mathbf{F}^H\mathbf{\Lambda}^H(\mathbf{F}\mathbf{S}^H\mathbf{S}\mathbf{F}^H)\mathbf{\Lambda}\mathbf{F}+\mathbf{F}^H\mathbf{\Gamma}^H\mathbf{\Gamma}\mathbf{F})^{-1}(\mu \mathbf{K}^H\mathbf{S}^H \mathbf{b} +\mathbf{L}^H\mathbf{v})\,,
\end{align}
where $\mathbf{\Lambda},\,\mathbf{\Gamma}$ are defined in \eqref{eq:KL_diag}. Lemma \ref{lem:FSSH} provides a useful expression for the product  $(\mathbf{FS}^H\mathbf{SF}^H)$,  by which \eqref{eq:x_2} becomes:
\begin{align}
  \mathbf{x}^*(\mu) =&  \left(\frac{\mu}{d} \mathbf{F}^H\mathbf{\Lambda}^H\mathbf{P}^T(\mathbf{I}_n\otimes \mathbf{J}_d)\mathbf{P\Lambda}\mathbf{F}+\mathbf{F}^H\mathbf{\Gamma}^H\mathbf{\Gamma}\mathbf{F}\right)^{-1}(\mu \mathbf{K}^H\mathbf{S}^H \mathbf{b} +\mathbf{L}^H\mathbf{v})\\
  =&\mathbf{F}^H\left(\frac{\mu}{d}\mathbf{\Lambda}^H\mathbf{P}^T(\mathbf{I}_n\otimes \mathbf{J}_d)\mathbf{P\Lambda}+\mathbf{\Gamma}^H\mathbf{\Gamma}\right)^{-1}\mathbf{F}(\mu \mathbf{K}^H\mathbf{F}^H\mathbf{F}\mathbf{S}^H \mathbf{b} +\mathbf{L}^H\mathbf{F}^H\mathbf{F}\mathbf{v})\\
  =&\mathbf{F}^H\left(\frac{\mu}{d}\mathbf{\Lambda}^H\mathbf{P}^T(\mathbf{I}_n\otimes \mathbf{J}_d)\mathbf{P\Lambda}+\mathbf{\Gamma}^H\mathbf{\Gamma}\right)^{-1}(\mu \mathbf{\Lambda}^H\tilde{\mathbf{b}}_H +\mathbf{\Gamma}^H\tilde{\mathbf{v}})\,,  \label{eq:optimality1}
\end{align}
where  $\tilde{\mathbf{b}}_H=\mathbf{F} \mathbf{b}_H=\mathbf{F}\mathbf{S}^H\mathbf{b}$ contains $d$ replication of $\tilde{\mathbf{b}}$ - see, e.g., \cite{MilanfarSR}.
We now introduce the following operators
\begin{equation}
\label{eq:lam_bar}
    \underline{\mathbf{\Lambda}} := \left(\mathbf{I}_n\otimes \mathbf{1}_d^T\right)\mathbf{P\Lambda}  \qquad
    \underline{\mathbf{\Lambda}}^H := \mathbf{\Lambda}^H\mathbf{P}^T\left(\mathbf{I}_n\otimes \mathbf{1}_d\right)
\end{equation}
where $\mathbf{1}_d\in\R^d$ is a vector of ones. In compact form, equation \eqref{eq:optimality1} reads: 
\begin{equation}\label{eq:x_sol}
 \mathbf{x}^*(\mu)=\mathbf{F}^H\left(\frac{\mu}{d} \underline{\mathbf{\Lambda}}^H\underline{\mathbf{\Lambda}}+\mathbf{\Gamma}^H\mathbf{\Gamma}\right)^{-1}(\mu \mathbf{\Lambda}^H\tilde{\mathbf{b}}_H +\mathbf{\Gamma}^H\tilde{\mathbf{v}})\,.
\end{equation}
Proceeding as in \cite{FSR}, we can now apply the Woodbury formula \eqref{eq:woodbury} and perform few manipulations, so as to obtain that the expression in \eqref{eq:x_sol} becomes:
\begin{equation}
  \mathbf{x}^*(\mu) 
    =\mathbf{F}^H\left[\mathbf{\Psi}-\mu\mathbf{\Psi}\underline{\mathbf{\Lambda}}^H\left(d\mathbf{I}+\mu\underline{\mathbf{\Lambda}}\mathbf{\Psi}\underline{\mathbf{\Lambda}}^H\right)^{-1}\underline{\mathbf{\Lambda}}\mathbf{\Psi}\right](\mu \mathbf{\Lambda}^H\tilde{\mathbf{b}}_H +\mathbf{\Gamma}^H\tilde{\mathbf{v}})\,, \label{eq:x_sol_2}
\end{equation}
whence the Fourier transform of the high resolution residual $\mathbf{r}_H^*(\mu)=\mathbf{Kx^*}(\mu)-\mathbf{b}$, with $\mathbf{x}^*(\mu)$ given in \eqref{eq:x_sol_2}, can be written as
\begin{equation}
  \tilde{\mathbf{r}}^*_H(\mu) 
%
=\mathbf{\Lambda}\left[\mathbf{\Psi}-\mu\mathbf{\Psi}\underline{\mathbf{\Lambda}}^H\left(d\mathbf{I}+\mu\underline{\mathbf{\Lambda}}\mathbf{\Psi}\underline{\mathbf{\Lambda}}^H\right)^{-1}\underline{\mathbf{\Lambda}}\mathbf{\Psi}\right](\mu \mathbf{\Lambda}^H\tilde{\mathbf{b}}_H +\mathbf{\Gamma}^H\tilde{\mathbf{v}})-\tilde{\mathbf{b}}_H\,, \label{eq:FRH_2}
\end{equation}
where $\mathbf{\Psi} = (\mathbf{\Gamma}^H\mathbf{\Gamma}+\epsilon)^{-1}$ and the parameter $0<\epsilon\ll1$ guarantees the inversion of $\mathbf{\Gamma}^H\mathbf{\Gamma}$.
Recalling Lemma \ref{lem:kron} and the property \eqref{eq:permutation}, we prove the following result.

\begin{proposition}\label{lem:new}
Let $\mathbf{\Phi}\in\R^{n\times n}$ be a diagonal matrix and consider the matrix $\underline{\mathbf{\Lambda}}$ defined in \eqref{eq:lam_bar}. Then, the following equality holds:
\begin{equation}\label{eq:dec}
\underline{\mathbf{\Lambda}}^H \mathbf{\Phi} \underline{\mathbf{\Lambda}} = \mathbf{P}^T(\mathbf{\Phi}\otimes\mathbf{I}_d)\mathbf{P}\underline{\mathbf{\Lambda}}^H\underline{\mathbf{\Lambda}}\,.
\end{equation}
\end{proposition}
\begin{proof}
Recalling property \eqref{eq:kron} in Lemma \ref{lem:kron}, we get the following chain of equalities
\begin{align}
    \underline{\mathbf{\Lambda}}^H \mathbf{\Phi} \underline{\mathbf{\Lambda}}=& \mathbf{\Lambda}^H\mathbf{P}^T(\mathbf{I}_n\otimes \mathbf{1}_d)\mathbf{\Phi}(\mathbf{I}_n\otimes \mathbf{1}_d^T)\mathbf{P\Lambda}   =\mathbf{\Lambda}^H\mathbf{P}^T(\mathbf{I}_n\otimes \mathbf{1}_d)(\mathbf{\Phi}\otimes \mathbf{1}_d^T)\mathbf{P\Lambda}     \\
    =&\mathbf{\Lambda}^H\mathbf{P}^T(\mathbf{I}_n\mathbf{\Phi}\otimes \mathbf{1}_d\mathbf{1}_d^T)\mathbf{P\Lambda}=\mathbf{\Lambda}^H\mathbf{P}^T(\mathbf{\Phi}\mathbf{I}_n\otimes \mathbf{J}_d)\mathbf{P\Lambda}\\
        =&\mathbf{\Lambda}^H\mathbf{P}^T(\mathbf{\Phi}\mathbf{I}_n\otimes\mathbf{I}_d \mathbf{J}_d)\mathbf{P\Lambda}=\mathbf{\Lambda}^H\mathbf{P}^T(\mathbf{\Phi}\otimes\mathbf{I}_d)(\mathbf{I}_n\otimes \mathbf{J}_d)\mathbf{P\Lambda}\\
      =&\mathbf{\Lambda}^H\mathbf{P}^T(\mathbf{\Phi}\otimes\mathbf{I}_d)\mathbf{P}\mathbf{P}^T(\mathbf{I}_n\otimes \mathbf{J}_d)\mathbf{P\Lambda},
\end{align}
where the sparse block-diagonal matrix $\mathbf{P}^T(\mathbf{\Phi}\otimes\mathbf{I}_d)\mathbf{P}\in\mathbb{R}^{N\times N}$ commutes with $\mathbf{\Lambda}^H$, so that  $\mathbf{\Lambda}^H\mathbf{P}^T(\mathbf{\Phi}\otimes\mathbf{I}_d)\mathbf{P}=\mathbf{P}^T(\mathbf{\Phi}\otimes\mathbf{I}_d)\mathbf{P}\mathbf{\Lambda}^H$. Recalling \eqref{eq:lam_bar}, this yields:
\begin{equation}
      \underline{\mathbf{\Lambda}}^H \mathbf{\Phi} \underline{\mathbf{\Lambda}} =
       \mathbf{P}^T(\mathbf{\Phi}\otimes\mathbf{I}_d)\mathbf{P}\underline{\mathbf{\Lambda}}^H\underline{\mathbf{\Lambda}}\,,
\end{equation}
which completes the proof. $\qedsymbol$
\end{proof}

\begin{cor}\label{cor:1}
Let $\mathbf{\Phi}=\big(d\mathbf{I}+\mu\underline{\mathbf{\Lambda}}\mathbf{\Psi}\underline{\mathbf{\Lambda}}^H\big)^{-1}$. Then, the expression in  \eqref{eq:FRH_2} \mbox{turns into}
\begin{equation}
    \tilde{\mathbf{r}}^*_H(\mu) = \mathbf{\Lambda}\left[\mathbf{\Psi}-\mu
    \mathbf{\Psi P}^T\left((d\mathbf{I}+\mu   \underline{\mathbf{\Lambda}}\mathbf{\Psi}   \underline{\mathbf{\Lambda}}^H )^{-1}\otimes\mathbf{I}_d\right)\mathbf{P}\underline{\mathbf{\Lambda}}^H\underline{\mathbf{\Lambda}}\mathbf{\Psi}\right](\mu \mathbf{\Lambda}^H\tilde{\mathbf{b}}_H +\mathbf{\Gamma}^H\tilde{\mathbf{v}})-\tilde{\mathbf{b}}_H\,.
\end{equation}
\end{cor}
\begin{proof}
We first notice that
\begin{equation}\label{eq:LPL}
\underline{\mathbf{\Lambda}}\mathbf{\Psi}\underline{\mathbf{\Lambda}}^H= (\mathbf{I}_n\otimes \mathbf{1}_d^T)\widehat{\mathbf{\Lambda\Psi\Lambda}^H}(\mathbf{I}_n\otimes \mathbf{1}_d)\,,
\end{equation}
is diagonal as $\widehat{\mathbf{\Lambda\Psi\Lambda}^H}=\mathbf{P\Lambda\Psi\Lambda}^H\mathbf{P}^T$ is. \mbox{The matrix in \eqref{eq:LPL} can thus be written as}
\begin{equation}\label{eq:omega}
\underline{\mathbf{\Lambda}}\mathbf{\Psi}\underline{\mathbf{\Lambda}}^H=\mathrm{diag}(\omega_1,\ldots,\omega_n),\qquad \omega_i = \displaystyle{\sum_{j=0}^{d-1}}\frac{|\hat{\lambda}_{\iota+j}|^2}{|\hat{\gamma}_{\iota+j}|^2+\epsilon}\,.
\end{equation}
Hence, since $\mathbf{\Phi}$ is the inverse of the sum of two diagonal matrices, it is diagonal so we can apply Proposition \ref{lem:new} and deduce the thesis.  $\qedsymbol$
\end{proof}
Recalling now the action of the  permutation matrix $\mathbf{P}$ on vectors, we have that the product $\hat{\mathbf{r}}_H^*(\mu)=\mathbf{P}\tilde{\mathbf{r}}_H^{*}(\mu)$ reads
\begin{equation}\label{eq:actg}
\!\!\hat{\mathbf{r}}_H^*(\mu){=} 
   \left[  \widehat{\mathbf{\Lambda}\mathbf{\Psi}}{-}\mu  \widehat{\mathbf{\Lambda}\mathbf{\Psi}}\left((d\mathbf{I}{+}\mu   \underline{\mathbf{\Lambda}}\mathbf{\Psi}   \underline{\mathbf{\Lambda}}^H )^{-1}\otimes\mathbf{I}_d\right)\widehat{\underline{\mathbf{\Lambda}}^H\underline{\mathbf{\Lambda}}\mathbf{\Psi}}\right](\mu \check{\mathbf{\Lambda}}\hat{\mathbf{b}}_H {+}\check{\mathbf{\Gamma}}\hat{\mathbf{v}})
    {-}\hat{\mathbf{b}}_H,
\end{equation}
where the matrix 
$\widehat{\underline{\mathbf{\Lambda}}^H\underline{\mathbf{\Lambda}}\mathbf{\Psi}}=    \mathbf{P}\mathbf{\Lambda}^H\mathbf{P}^T(\mathbf{I}_n\otimes \mathbf{J}_d)\mathbf{P\Lambda\Psi P}^T
$
acts on $\mathbf{g}\in\R^N$ as 
\begin{equation}
    (\widehat{\underline{\mathbf{\Lambda}}^H\underline{\mathbf{\Lambda}}\mathbf{\Psi}}\mathbf{g})_i = \bar{\hat{\lambda}}_i \displaystyle{\sum_{j=0}^{d-1}}\frac{\hat{\lambda}_{\iota+j}}{|\hat{\gamma}_{\iota+j}|^2+\epsilon}g_{\iota+j}\,.
\end{equation}
Combining altogether, we finally deduce:
\begin{align}
  \hat{\mathbf{r}}_H^*(\mu) =& \mu  \widehat{\mathbf{\Lambda\Psi }}\check{\mathbf{\Lambda}}\hat{\mathbf{b}}_H+ \widehat{\mathbf{\Lambda\Psi }}\check{\mathbf{\Gamma}}\hat{\mathbf{v}}
  -\mu^2 \widehat{\mathbf{\Lambda\Psi }}\left[(d\mathbf{I}+\mu   \underline{\mathbf{\Lambda}}\mathbf{\Psi}   \underline{\mathbf{\Lambda}}^H )^{-1}\otimes\mathbf{I}_d\right]\widehat{\underline{\mathbf{\Lambda}}^H\underline{\mathbf{\Lambda}}\mathbf{\Psi}}\check{\mathbf{\Lambda}} \hat{\mathbf{b}}_H\\
   -&\mu\widehat{\mathbf{\Lambda\Psi }}\left[(d\mathbf{I}+\mu   \underline{\mathbf{\Lambda}}\mathbf{\Psi}   \underline{\mathbf{\Lambda}}^H )^{-1}\otimes\mathbf{I}_d\right]\widehat{\underline{\mathbf{\Lambda}}^H\underline{\mathbf{\Lambda}}\mathbf{\Psi}} \check{\mathbf{\Gamma}}\hat{\mathbf{v}}-\hat{\mathbf{b}}_H\,,
\end{align}
whence we can explicitly compute the expression for each component $i=1,\ldots,n$:
\begin{align}
\left(\hat{\mathbf{r}}_H^*(\mu)\right)_i=&\mu\left[\frac{|\hat{\lambda_i}|^2}{|\hat{\gamma_i}|^2+\epsilon}\,\hat{b}_{H,i}\right] + \frac{\hat{\lambda_i}\bar{\hat{\gamma_i}}\hat{v}_i}{|\hat{\gamma_i}|^2+\epsilon}- \left[\mu^2\displaystyle{\sum_{j=0}^{d-1}\frac{|\hat{\lambda}_{\iota+j}|^2\hat{b}_{H,\iota+n}}{|\hat{\gamma}_{\iota+j}|^2+\epsilon}}\right.\\
+&\left.\mu\displaystyle{\sum_{j=0}^{d-1}\frac{\hat{\lambda}_{\iota+j}\bar{\hat{\gamma}}_{\iota+j}\hat{v}_{\iota+j}}{|\hat{\gamma}_{\iota+j}|^2+\epsilon}}\right]\frac{|\hat{\lambda}_i|^2}{|\hat{\gamma_i}|^2+\epsilon}\left(d+\mu\displaystyle{\sum_{j=0}^{d-1}}\frac{|\hat{\lambda}_{\iota+j}|^2}{|\hat{\gamma}_{\iota+j}|^2+\epsilon}\right)^{-1}-\hat{b}_{H,i}\,.
\end{align}
We can thus deduce the following expression of the terms in formula \eqref{eq:white_new}
\begin{align}\label{eq:PRH}
\begin{split}
&\displaystyle{\sum_{j=0}^{d-1}}(\hat{\mathbf{r}}_H^*(\mu))_{\iota+j} =\frac{1}{d+\mu\displaystyle{\sum_{j=0}^{d-1}}\frac{|\hat{\lambda}_{\iota+j}|^2}{|\hat{\gamma}_{\iota+j}|^2+\epsilon}}\Bigg[\mu \Bigg(
d \displaystyle{\sum_{j=0}^{d-1}}\frac{|\hat{\lambda}_{\iota+j}|^2}{|\hat{\gamma}_{\iota+j}|^2+\epsilon}\,\hat{b}_{H,{\iota+j}} \\
&-  \displaystyle{\sum_{j=0}^{d-1}}\hat{b}_{H,\iota+j} \displaystyle{\sum_{j=0}^{d-1}}\frac{|\hat{\lambda}_{\iota+j}|^2}{|\hat{\gamma}_{\iota+j}|^2+\epsilon} \Bigg)
+d\left( \displaystyle{\sum_{j=0}^{d-1}}\frac{\hat{\lambda}_{\iota+j}\bar{\hat{\gamma}}_{\iota+j}\hat{v}_{\iota+j}}{|\hat{\gamma}_{\iota+j}|^2+\epsilon}-\displaystyle{\sum_{j=0}^{d-1}}\hat{b}_{H,\iota+j}\right) \Bigg]\,.\end{split}
\end{align}

In light of its replicating structure, we observe that the action of the permutation $\mathbf{P}$ on $\tilde{\mathbf{b}}_H$ will cluster the identical entries, so that the $\hat{b}_{H,\iota+j}$ can be written as the mean of the set of $d$ values $\{\hat{b}_{H,\iota},\ldots,\hat{b}_{H,\iota+d-1}\}$. This allows to simplify formula \eqref{eq:PRH} as the difference in the first bracket vanishes. By now setting
\begin{equation}
    \eta_i :=\frac{1}{d}\displaystyle{\sum_{j=0}^{d-1}}\frac{|\hat{\lambda}_{\iota+j}|^2}{|\hat{\gamma}_{\iota+j}|^2+\epsilon},\quad\varrho_i := \displaystyle{\sum_{j=0}^{d-1}}\hat{b}_{H,\iota+j},\quad\nu_i := \displaystyle{\sum_{j=0}^{d-1}}\frac{\hat{\lambda}_{\iota+j}\bar{\hat{\gamma}}_{\iota+j}\tilde{v}_{\iota+j}}{|\hat{\gamma}_{\iota+j}|^2+\epsilon}\,,
\end{equation}
which can all be computed beforehand. Plugging \eqref{eq:PRH} into \eqref{eq:white_new}  we finally get
\begin{equation} \label{eq:white_fin}
W(\mu) = \left(\displaystyle{\sum_{i=1}^{N}\left|\frac{\nu_i-\varrho_i}{1+\eta_i\mu}\right|^4   }\right) \Big/ {\left(\displaystyle{\sum_{i=1}^{N}\left|\frac{\nu_i-\varrho_i}{1+\eta_i\mu}\right|^2 }\right)^2  }\,.  
\end{equation}
Note that when $d=1$, i.e. when no decimation is considered, this formula corresponds exactly to the one considered in \cite{etna} in the context of image deblurring.

According to the RWP, the optimal $\mu^*$ is selected as the one minimising the whiteness measure function in \eqref{eq:white_fin}. We remark that the action of the permutation matrix $\mathbf{P}$ can be efficiently replicated without deriving its explicit expression; as a result, the overall computational cost for the evaluation of $W(\mu)$ amounts to $O(N \log N)$, namely the cost of the 2D fast Fourier transform and of its inverse, and the value $\mu^*$ can be efficiently detected via grid-search. Finally, the optimal $\mu^*$ is used for the computation of the reconstruction $\mathbf{x}^*(\mu^*)$ based on \eqref{eq:x_sol_2}.

The main steps of the proposed procedure are summarised in Algorithm~\ref{alg:1}.


\begin{algorithm}[H]\small
\caption{SR for \eqref{eq:l2l2} with automatic parameter selection via RWP}
\vspace{0.2cm}
		{\renewcommand{\arraystretch}{1.2}
			\renewcommand{\tabcolsep}{0.0cm}
			\vspace{-0.08cm}
			\begin{tabular}{lll}	\multicolumn{2}{l}{\textbf{inputs}:} & $\;$
				observed image $\mathbf{b}\in\R^n$, forward model operator $\mathbf{K}\in\R^{nd\times nd}$,\\
				\multicolumn{2}{l}{\phantom{\textbf{inputs}:}} & $\;$ down-sampling operator $\mathbf{S}\in\R^{n\times nd}$ \\	
				\vspace{0.2cm}
					\begin{tabular}{lll}
				\textbf{$\bullet$} & \multicolumn{2}{l}{$\;$\textbf{ Compute Fourier diagonalisations}\textbf{:}$\;\mathbf{\Lambda}=\mathbf{F K F}^H,\,\mathbf{\Gamma}=\mathbf{F L F}^H$}\\
				\textbf{$\bullet$} & \multicolumn{2}{l}{$\;$\textbf{ Compute matrices}\textbf{:}$\;\underline{\mathbf{\Lambda}}{=}(\mathbf{I}_n\otimes \mathbf{1}_d^T)\mathbf{P\Lambda},\,\mathbf{\Psi}{=}(\mathbf{\Gamma}^H\mathbf{\Gamma}+\epsilon)^{-1}$}\\
				\textbf{$\bullet$} & \multicolumn{2}{l}{$\;\;\,$\textbf{Residual whiteness principle for the selection of $\mu^*$ }\textbf{:}}\\
				 & \multicolumn{2}{l}{$\;\bf{\cdot}$ Compute $W(\mu)$ in \eqref{eq:white_fin} for different values of $\mu$, based on Corollary \ref{cor:1} and \eqref{eq:actg}}\\
				& \multicolumn{2}{l}{$\;\bf{\cdot}$ Select $\mu^*\in\argmin W(\mu)$}\\
				\textbf{$\bullet$}&	\multicolumn{2}{l}{$\;\;\,$\textbf{Compute the reconstruction}\textbf{:}$\;\mathbf{x}^*(\mu^*)$ by \eqref{eq:x_sol_2}}\\
				\end{tabular}
			\end{tabular}
		
		}
		\label{alg:1}
\end{algorithm}

\section{Numerical results}
\label{sec:ex}

We evaluate the proposed RWP-based automatic procedure for selecting the regularisation parameter $\mu$
in variational models of the form \eqref{eq:l2l2} when \mbox{ $\mathbf{v}=\mathbf{0}_N$} and $\mathbf{L} = \mathbf{D} \;{:=}\;\left(\mathbf{D}_h^T,\mathbf{D}_v^T\right)^T 
{\in} \R^{2N \times N}$, with $\mathbf{D}_h,\mathbf{D}_v\in \R^{N \times N}$ representing the finite difference operators discretising the first-order horizontal and vertical partial derivatives, respectively. Note that $\mathbf{D}$ verifies assumptions (A3)-(A4) in Sec. \ref{sec:ass}.

Our goal is to highlight that the
proposed RWP selects a regularisation parameter
value $\mu^*$ yielding high quality restorations. The RWP is compared with the DP, defined in \eqref{eq:dp} when $\tau=1$. There is a one-to-one relationship between the $\mu$-value and the norm of the associated residual image. Hence, in all the presented results we will substitute the $\mu$-values with the corresponding $\tau$-values, with $\tau$ defined according to (\ref{eq:dp}) by $\tau^*(\mu) \;{:=}\; \lVert \mathbf{SHx}^*(\mu) - \mathbf{b}  \rVert_2/(\sqrt{n}\sigma)$.

The quality of the restorations $\mathbf{x}^*$, for different values of $\tau^*$, with respect to the original
undecimated image $\mathbf{x}$, will be assessed by means of three scalar measures, namely the Structural Similarity Index (SSIM) \cite{ssim}, the Peak-Signal-to-Noise-Ratio (PSNR) and the Improved-Signal-to-Noise Ratio (ISNR), defined by PSNR= $20\log_{10}(\sqrt{N}\max(\mathbf{x},\mathbf{x}^*)/\|\mathbf{x}-\mathbf{x}^*\|_2)$ and ISNR=$ 10\log_{10}(\|\mathbf{x}-\bar{\mathbf{b}}\|_2/\|\mathbf{x}-\mathbf{x}^*\|_2)$, respectively, with $\max(\mathbf{x},\mathbf{x}^*)$ representing the largest value of $\mathbf{x}$ and $\mathbf{x}^*$, while $\bar{\mathbf{b}}$ denotes the bicubic interpolation of $\mathbf{b}$.

\begin{figure}
\begin{minipage}[c]{0.59\textwidth}
\small{
    \centering
    \begin{subfigure}[t]{0.32\textwidth}
    \centering
\begin{overpic}[width=0.91in]{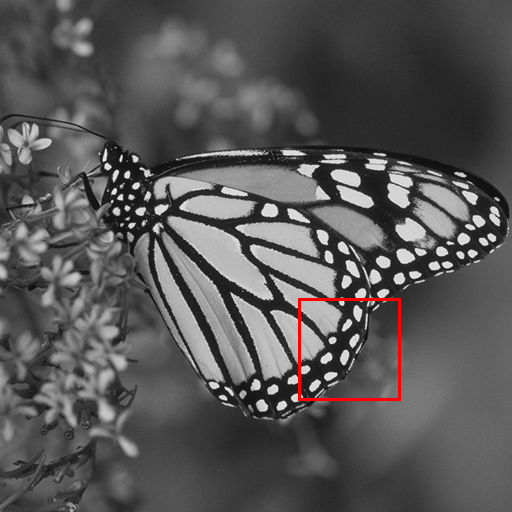}
\put(0.5,41.5){\color{red}%
\frame{\includegraphics[scale=.098]{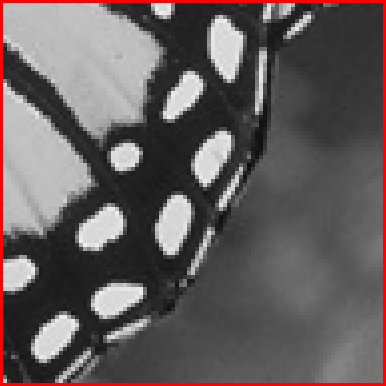}}}
\end{overpic}    
\caption{$\mathbf{x}$}
   \label{fig:mon_tr}
    \end{subfigure}
    \begin{subfigure}[t]{0.32\textwidth}
    \centering
\begin{overpic}[width=0.91in]{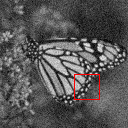}
\put(0.5,41){\color{red}%
\frame{\includegraphics[scale=.392]{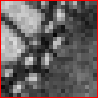}}}
\end{overpic}    
\caption{$\mathbf{b}$(x4)}
   \label{fig:mon_data}
    \end{subfigure}
    \begin{subfigure}[t]{0.32\textwidth}
    \centering
\begin{overpic}[width=0.91in]{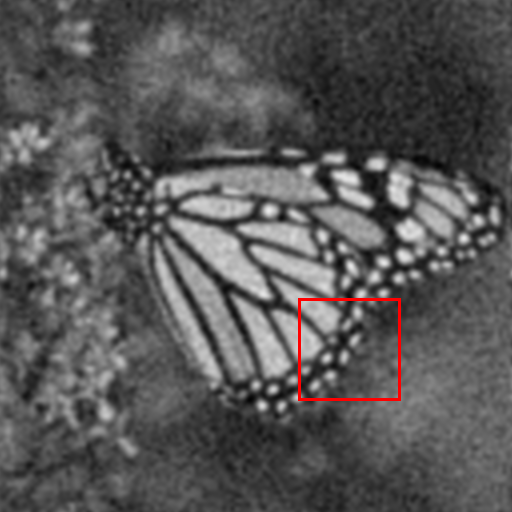}
\put(0.5,41.5){\color{red}%
\frame{\includegraphics[scale=.098]{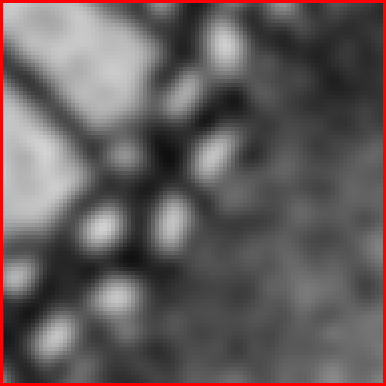}}}
\end{overpic}    
\caption{$\mathbf{x}^*(\mu^*)$}
   \label{fig:mon_rec}
    \end{subfigure}\\
    \begin{subfigure}[t]{0.32\textwidth}
    \centering
\begin{overpic}[width=0.91in]{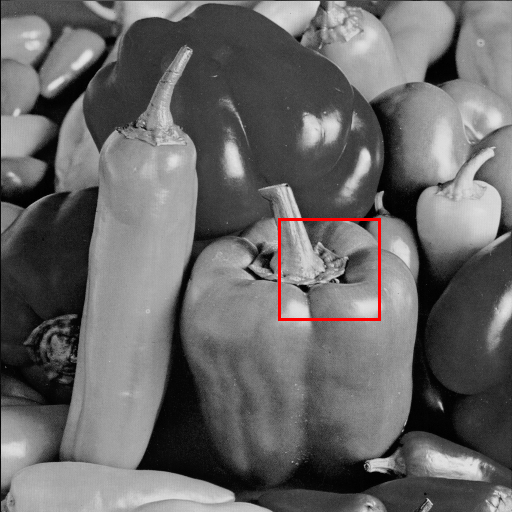}
\put(0.5,41.5){\color{red}%
\frame{\includegraphics[scale=.098]{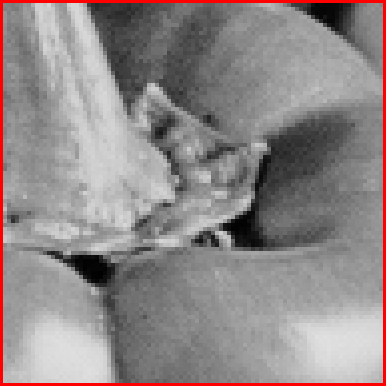}}}
\end{overpic}    
\caption{$\mathbf{x}$}
   \label{fig:pepp_tr}
    \end{subfigure}
    \begin{subfigure}[t]{0.32\textwidth}
    \centering
\begin{overpic}[width=0.91in]{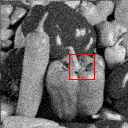}
\put(0.5,41){\color{red}%
\frame{\includegraphics[scale=.392]{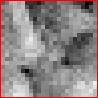}}}
\end{overpic}    
\caption{$\mathbf{b}$(x4)}
   \label{fig:pepp_data}
    \end{subfigure}
    \begin{subfigure}[t]{0.32\textwidth}
    \centering
\begin{overpic}[width=0.91in]{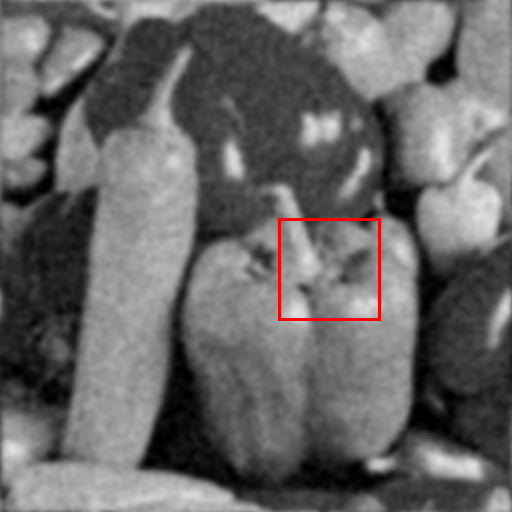}
\put(0.5,41.5){\color{red}%
\frame{\includegraphics[scale=.098]{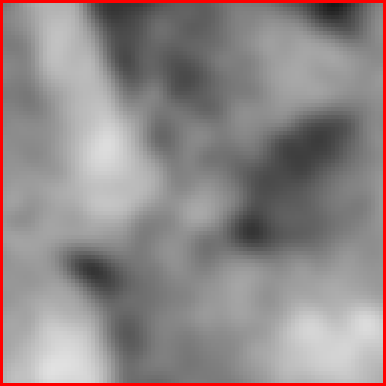}}}
\end{overpic}    
\caption{$\mathbf{x}^*(\mu^*)$}
   \label{fig:pepp_rec}
    \end{subfigure}
    \caption{From left to right: original $\mathbf{x}$, observed $\mathbf{b}$, our reconstruction $\mathbf{x}^*(\mu^*)$ for \texttt{monarch} (top) and \texttt{peppers} (bottom).}
     }
     \end{minipage}
\hfill
\begin{minipage}[c]{0.39\textwidth}\null
{\small
\setlength{\tabcolsep}{2.1pt}
\renewcommand{\arraystretch}{1.3}	
\centering
\begin{tabular}{l|ccc||c}
\multicolumn{4}{c}{\qquad\texttt{monarch}}&\\
\hline\hline
&PSNR&ISNR&SSIM&\\
\hline\hline
$\mathbf{x}^*$&21.4050&1.3452& 0.6736&
\multirow{2}{*}{\STAB{\rotatebox[origin=c]{90}{Test 1}}}\\
$\bar{\mathbf{b}}$&20.0598&-&0.5435&\\
\hline
$\mathbf{x}^*$&18.9277&1.1561&0.6297&
\multirow{2}{*}{\STAB{\rotatebox[origin=c]{90}{Test 2}}}\\
$\bar{\mathbf{b}}$&17.7716&-&0.3105&\\
\hline
\multicolumn{4}{c}{\qquad\texttt{peppers}}&\\
\hline\hline
&PSNR&ISNR&SSIM&\\
\hline\hline
$\mathbf{x}^*$&23.5674&1.9147&0.6757 &
\multirow{2}{*}{\STAB{\rotatebox[origin=c]{90}{Test 1}}}\\
$\bar{\mathbf{b}}$&21.6526&-&0.5187&\\
\hline
$\mathbf{x}^*$&21.3034&2.6078&0.6240 &
\multirow{2}{*}{\STAB{\rotatebox[origin=c]{90}{Test 2}}}\\
$\bar{\mathbf{b}}$&18.6956&-&0.3032&\\
\hline
\end{tabular}
\captionof{table}{Achieved PSNR, ISNR, SSIM values for \texttt{monarch} and \texttt{peppers} for the two degradation settings.}\label{tab:vals 1}
}
\end{minipage}
\end{figure}
We consider two test images of size $512\times 512$ with pixel values normalised in $[0,1]$, namely \texttt{monarch} and \texttt{peppers}, shown in Figs.~\ref{fig:mon_tr}-\ref{fig:pepp_tr}, respectively. The decimation factors along the rows and the columns of the original images are set as $d_c=d_r=4$. 
As a first example, the original test images are corrupted by Gaussian blur, generated by the Matlab routine \texttt{fspecial} with input parameters \texttt{band}=9 and \texttt{sigma}=2. The \texttt{band} parameter represents the side length (in pixels) of the square support of the kernel, whereas \texttt{sigma} is the standard deviation (in pixels) of the isotropic bivariate Gaussian distribution defining the kernel in the continuous setting. Finally, the decimated and blurred images are corrupted by an AWGN with standard deviation $\sigma=0.05$. The observed data for the test images \texttt{monarch} and \texttt{peppers} are displayed in Fig. \ref{fig:mon_data}-\ref{fig:pepp_data}, respectively.\\
\indent In Figs. \ref{fig:ws}, we report the behavior of the whiteness measure $W(\mu)$ as a function of $\tau^*(\mu)$ for the test images \texttt{monarch} (solid blue line) and \texttt{peppers} (solid black line), respectively. The plotted values have been obtained by solving the model \eqref{eq:l2l2} for a fine grid of different $\mu$ values, and then computing for each $\mu$ the
associated $\tau^*(\mu)$ and $W(\mu)$. The optimal $\tau^*$s corresponding to $\mu^*$s are indicated by the vertical dashed magenta and green lines for \texttt{monarch} ($\tau^*(\mu^*)=0.9398$) and \texttt{peppers} ($\tau^*(\mu^*)=0.9633$), respectively, while $\tau=1$, representing the DP, is depicted by the vertical dotted black line. Notice that the whiteness curves computed \emph{a posteriori} admit a minimiser over the considered domain which coincides with the $\tau^*$ selected by the RWP. Moreover, the proximity of the optimal $\tau^*$s to $1$ indicates that the noise level estimated starting from $\mathbf{r}^*(\mu)$ is close to the true one.\\ 
\indent In Figs.~\ref{fig:val_mon}-\ref{fig:val_pepp}, we graphic the achieved ISNR and SSIM for the two test images. Note that the RWP tends to automatically select a $\mu$-value returning the best trade-off between the two quality measures. The reconstructed $\mathbf{x}^*(\mu^*)$ for the two test images are shown in Figs. \ref{fig:mon_rec}-\ref{fig:pepp_rec}. Finally, the PSNR, ISNR and SSIM values achieved by the proposed strategy are reported in the top part of Tab.~\ref{tab:vals 1} (Test 1), together with the PSNR and SSIM of the bicubic interpolation.\\
\indent As a second example, we perform the same reconstructions with different degradation levels. More specifically, we consider a Gaussian blur with parameters \texttt{band} = 13, \texttt{sigma} = 3, and AWGN with standard deviation $\sigma=0.1$. The achieved quality measures are reported in the bottom part of Tab.\ref{tab:vals 1} (Test 2). In this case, the RWP automatically selects an optimal $\tau^*$ corresponding to a very accurate estimate of the original noise standard deviation, namely $\tau^*(\mu^*)=0.9938$ for \texttt{monarch} and $\tau^*(\mu^*)=0.9915$ for \texttt{peppers}.

\begin{figure}
    \centering
    \begin{subfigure}[t]{0.325\textwidth}
    \centering
    \includegraphics[height=1.12in]{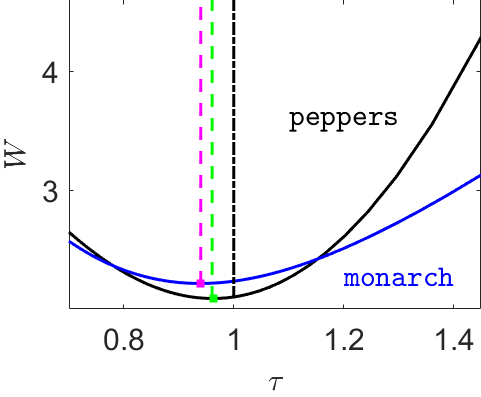}
        \caption{}
    \label{fig:ws}
    \end{subfigure}
    \begin{subfigure}[t]{0.325\textwidth}
    \centering
    \includegraphics[height=1.12in]{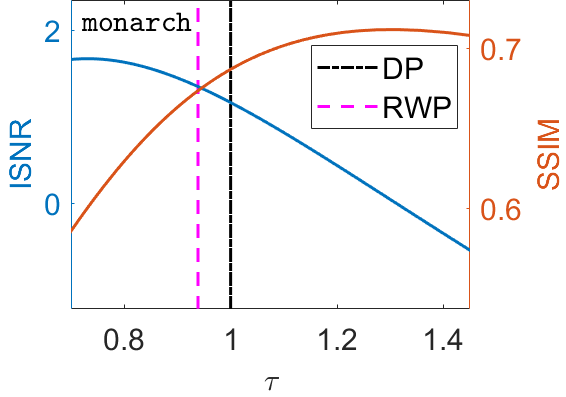}
            \caption{}
        \label{fig:val_mon}
    \end{subfigure}
    \begin{subfigure}[t]{0.325\textwidth}
    \centering
    \includegraphics[height=1.12in]{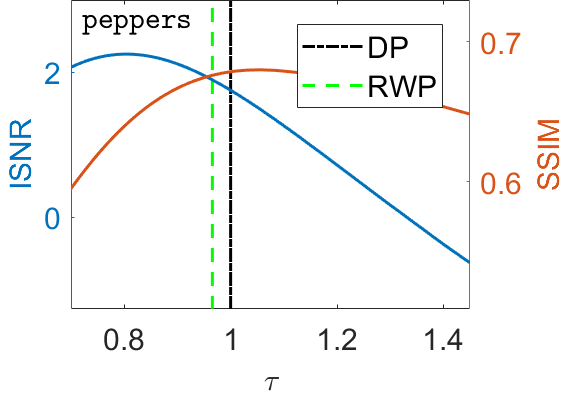}
            \caption{}
        \label{fig:val_pepp}
    \end{subfigure}
\caption{Whiteness measure functions (left column) and ISNR/SSIM values for different $\tau$s (second column) for the \texttt{monarch} and \texttt{peppers} test images.}
\end{figure}
\section{Conclusions and outlook}
We extended the residual whiteness principle introduced in \cite{etna} for automatic parameter selection with unknown noise level in image deblurring to image super-resolution problems solved by generalised Tikhonov regularisation models in the form \eqref{eq:l2l2} whose solution can be efficiently computed by means of the approach outlined in \cite{FSR}. By exploiting carefully technical properties of the operators involved in the model in the frequency domain, a compact  formula for the whiteness measure can be found. Its minimisation provides an accurate estimate of the unknown noise level. As a future work, we plan to explicitly formalise the extension of the RWP to non-smooth super-resolution models as well as to explicitly tackle the minimisation of the whiteness measure with more sophisticated techniques.

\bibliographystyle{plain}


\end{document}